\documentclass[12pt,oneside,a4paper]{article}
\usepackage{amssymb,verbatim,enumerate,ifthen}
\usepackage{cite}
\usepackage[mathscr]{eucal}
\usepackage[utf8]{inputenc}
\usepackage[T1]{fontenc}
\usepackage{graphicx}
\usepackage{amsthm,amsmath}
\usepackage{tikz}

\usepackage[marginparwidth=2.4cm]{geometry}

\newtheorem{thm}{Theorem}[section]
\newtheorem{prop}[thm]{Proposition}

\newtheorem{lem}[thm]{Lemma}
\newtheorem{cor}[thm]{Corollary}

\theoremstyle{definition}
\newtheorem{ex}[thm]{Example}

\theoremstyle{definition}

\newtheorem{rem}[thm]{Remark}

\newcommand{\mr}{\mathbb R}
\newcommand{\ceil}[1]{\lceil #1 \rceil}
\newcommand{\n}{\mathbb N}
\newcommand{\mq}{\mathbb Q}

\newcommand{\ra}{\rightarrow}
\newcommand{\Tgg}{T_{\gamma_-,\gamma_+}}
\newcommand{\Tggn}{T_{\gamma_-,\gamma_+,n}}
\newcommand{\Taa}{T_{\alpha_-,\alpha_+}}
\newcommand{\gp}{\gamma_+}
\newcommand{\gm}{\gamma_-}
\newcommand{\zn}{(0,+\infty)}
\newcommand{\jn}{(1,+\infty)}
\newcommand{\w}{\textnormal{int}}
\newcommand{\cl}{\textnormal{cl}}

\begin{document}
\title{Extension theorem for simultaneous q-difference equations and some its consequences}

\author{Witold Jarczyk and Pawe{\l} Pasteczka}

\date{}

\maketitle
\begin{abstract} Given a set $T \subset (0, +\infty)$, intervals $I\subset (0, +\infty)$ and $J\subset \mr$, as well as functions $g_t:I\times J\ra J$ with $t$'s running through the set 
\[ T^{\ast}:=T \cup \big\{t^{-1}\colon t \in T\big\}\cup\{1\} \]
we study the simultaneous $q$-difference equations
\[
\varphi(tx)=g_t\left(x,\varphi(x)\right), \qquad t \in T^{\ast},
\]
postulated for $x \in I\cap t^{-1}I$; here the unknown function $\varphi$ is assumed to map $I$ into $J$. We prove an Extension theorem stating that if $\varphi$ is continuous [analytic] on a nontrivial subinterval of $I$, then $\varphi$ is continuous [analytic] provided $g_t, t \in T^{\ast}$, are continuous [analytic]. The crucial assumption of the Extension theorem is formulated with the help of the so-called limit ratio $R_T$ which is a uniquely determined number from $[1,+\infty]$, characterising some density property of the set $T^{\ast}$. As an application of the Extension theorem we find the form of all continuous on a subinterval of $I$ solutions $\varphi:I \ra \mr$ of the simultaneous equations
\[
\varphi(tx)=\varphi(x)+c(t)x^p, \qquad t\in T,
\] 
where $c:T\ra \mr$ is an arbitrary function, $p$ is a given real number and $\sup I > R_T \inf I$.
\vspace{0,2cm}

\begin{flushleft}
{\bf Mathematics Subject Classification (2010).} Primary 39A13, 39B72; Secondary 39B12, 39B22.
\end{flushleft}

\begin{flushleft}
{\bf Keywords.} q-difference equations, Simultaneous equations, Equations on restricted domains, Extension of solutions, Form of continuous solutions. 
\end{flushleft}
\end{abstract}

\section*{Introduction} 
\noindent
Let $T \subset (0, +\infty)$ be an arbitrary nonempty set and $I\subset (0, +\infty)$ and $J \subset \mr$ be any intervals. Given functions $g_t$, $t\in T$, consider the simultaneous equations
\begin{eqnarray}\label{e1}
\varphi(tx)=g_t\left(x,\varphi(x)\right), \qquad t \in T,
\end{eqnarray}
postulated for $x \in I_t:= I\cap t^{-1}I$. Nonlinear equations \eqref{e1} are discussed in a great number of items cited in \cite[Chap. 3]{Ku} and \cite[Chap. 5]{KCG} (see also \cite[Sect. 1 and 7]{BJ}). For the first time simultaneous equations \eqref{e1} were studied in \cite{J}. They generalize several other systems (see, e.g., \cite{JP} and \cite[Lemma~6.3]{JJMM}). Repeating the argument used in \cite{JP} we see that it is enough to take into account only those $t$'s from $T$ that $I_t\neq \emptyset$. In what follows, for any $\varphi: I\ra \mr$ the phrase ''$\varphi$ is a solution of equations \eqref{e1}'' means ''$\varphi$ satisfies the equality $\varphi(tx)=g_t\left(x,\varphi(x)\right)$ for all $t \in T$ and $x \in I_t$''.

Notice that all systems of the form \eqref{e1} can be iterated (with careful restriction imposed on the domain of $x$). Motivated by this fact, in \cite{JP} we introduced the following notions which turned out to be useful. First define 
\[
T^{\ast}:=T\cup T^{-1}\cup \{1\},
\]
where $T^{-1}:=\left\{t^{-1}: t \in T\right\}$. Then for any numbers $\gm \in (0,1]$ , $\gp \in [1,+\infty)$ and $n \in \n$ put 
\begin{eqnarray*}
&&T_{\gm,\gp,n}:=\left\{t \in \zn:  \mbox{ there exist } t_1, \ldots, t_n \in T^{\ast} \, \mbox{ such that } \right. \\
&& \hspace{5cm}  t= t_1\cdot \ldots \cdot t_n \mbox{ and } \gm\leq t_1\cdot \ldots \cdot t_k \leq \gp    \\
&& \hspace{4.8cm} \left. \mbox{ for all } k=1, \ldots , n\right\}.
\end{eqnarray*}
Moreover, we denote
\[
\Tgg :=\bigcup^{\infty}_{n=1}T_{\gm,\gp,n}.
\]

It turns out that the density of $\Tgg$ is of crucial importance. For example, let us recall the main result from the paper \cite{JP} devoted to the simultaneous $q$-difference equations
\begin{eqnarray}\label{e2}
\varphi(tx)=\varphi(x)+c(t)x^p, \qquad t \in T.
\end{eqnarray}
This is a special case of \eqref{e1} with functions $g_t$, $t \in T$, given by $g_t(x,y)=y+c(t)x^p$. In what follows we accept the convention that given any $a \in (0,+\infty]$ the symbols $a/0$ and $\frac{a}{0}$ mean $+\infty$. This is quite natural as all variables $x$'s and parameters $t$'s appearing here run through sets of positive numbers only. 

The following result has been proved in \cite[Theorem]{JP}.

\vspace{0,2cm}

\noindent{\bf Theorem JP }{\it
Let $T \subset\zn$ be a nonempty set, $c: T\ra \mr$ and $I\subset \zn$ be a nontrivial interval. Assume that $\gm\in (0,1)$ and $\gp \in \jn$ are such that $\inf I/\gm<\sup  I/\gp$ and $\Tgg$ is a dense subset of the interval $\left[\gm, \gp\right]$. Let $\varphi:I\ra \mr$ be a continuous solution of equations \eqref{e2}.\\
\noindent
$(i)$ If $p\neq 0$ then there exist $a,b \in \mr$ such that
\begin{eqnarray}\label{e16}
\varphi(x)=ax^p+b, \qquad x \in \left(\inf I/\gm,\sup I/\gp\right).
\end{eqnarray}
\noindent
$(ii)$ If $p= 0$ then there exist $a,b \in \mr$ such that
\begin{eqnarray}\label{e17}
\varphi(x)=a \log x+b, \qquad x \in \left(\inf I/\gm,\sup I/\gp\right).
\end{eqnarray}
}

\vspace{0,2cm}

We will  generalize this result in few directions. First, the assumption of the density of $\Tgg$ is formulated in a weaker and more universal form. Secondly, we prove that it is sufficient to assume the continuity of $\varphi$ on a non trivial subinterval of $I$. Lastly, we show that the formulas for $\varphi$ are valid in the whole domain $I$, not only locally. All this is obtained with the help of Extension theorem proved in Section 2.

\section{Density of $\Tgg$}
In the paper \cite{JP} many results were  stated under the assumption  that $T\subset \zn$  is an arbitrary nonempty set and  $\gm\in (0,1)$, $\gp \in \jn$ are such that $\Tgg$ is a dense subset of $\left[\gm,\gp\right]$. There we also considered some particular cases, namely when 1 is an accumulation point of $T$, and when $T=\left\{a,b\right\} $, where $\log a/\log b \notin \mq$.  It turns out that in both  these cases the density of $\Tgg$ depends on the ratio  $\gp/\gm$. In this section we prove that this is a general rule. More precisely the following lemma is valid. 

\begin{lem} \label{l1.1}
Let $T\subset \zn$ and let $\gm\in (0,1)$ and $\gp \in \jn$ be numbers such that $\Tgg$ is a dense subset of $\left[\gm,\gp\right]$. Then $\Taa$ is a dense subset of $\left[\alpha_-,\alpha_+\right]$ for every $\alpha_-\in (0,1)$ and $\alpha_+ \in \jn$ such that $\alpha_+/\alpha_->\gp/\gm$.    
\end{lem}
\begin{proof}
Suppose to the contrary that $\left[\alpha_-,\alpha_+\right] \setminus \cl \Taa \neq \emptyset$  for some  $\alpha_-\in (0,1)$ and $\alpha_+ \in \jn$ satisfying the inequality $\alpha_+/\alpha_->\gp/\gm$. Then there exist a countable set $\Lambda$  and a family $\left\{I_{\lambda}\right\}_{\lambda \in \Lambda} $ of pairwise disjoint intervals such that
\[
\left(\alpha_-,\alpha_+\right)\setminus \cl \Taa=\bigcup_{\lambda \in \Lambda}I_{\lambda};
\]
their endpoints belong to $\left\{\alpha_-,\alpha_+\right\} \cup \cl \Taa$. As $\Taa\ne \emptyset$, at least one of those endpoints, say $a$, lies in $\left(\alpha_-,\alpha_+\right)$. Assume, for instance, that it is the left endpoint of an interval $I_{\lambda}$. So $ a \in \left(\alpha_-,\alpha_+\right) \cap \cl \Taa$ and we can find a $\kappa \in \jn$ with the following properties:
\begin{subequations}
\begin{alignat}{4}
\label{e3a}
\left(a,\kappa^2 a\right)&\cap \Taa= \emptyset, \\[2mm]
\label{e3b}
\kappa&<\gamma^{-1}_-, \,\, \kappa<\gp ,\\
\label{e3c}
\kappa^4&<\frac{\alpha_+\gm}{\alpha_-\gp},\\
\label{e3d}
\kappa^3&<\frac{\alpha_+}{a},\\
\label{e3e}
\kappa^2&<\frac{a}{\alpha_-}. 
\end{alignat}
\end{subequations}

Since $a \in \cl \Taa$ and \eqref{e3a} holds we can take $x_0\in \left(\kappa^{-1}a,a\right]\cap \Taa$. Then inequalities \eqref{e3d} and \eqref{e3e} imply
$$
\kappa^{-3}\alpha_+>a \ge x_0 \quad \text{ and }\quad \kappa\alpha_-<\kappa^{-1}a <x_0,
$$
respectively, and thus,
\begin{equation}\label{20220118a}
 x_0 \in (\kappa\alpha_-,\kappa^{-3}\alpha_+).
\end{equation}
Using \eqref{e3c}, \eqref{20220118a}, and the inequality $\gm<\gp$ we see that 
$$
\max\Big\{\frac{\kappa\alpha_-}{\gm},\frac{x_0}{\gp}\Big\}<\min\Big\{\frac{\kappa^{-3}\alpha_+}{\gp},\frac{x_0}{\gm}\Big\}.
$$
Take any element $m$ belonging to the open interval ending in these points. Then
$$m\gp>x_0,\quad m\gm<x_0$$
and
$$ \kappa\alpha_-<m\gm<m\gp<\kappa^{-3}\alpha_+.$$
Consequently, we come to 
\begin{subequations}
\begin{alignat}{4}
\label{e4a}
x_0&\in \left(m\gm,m\gp\right),\\
\label{e4b}
\left[m\gm,m\gp\right]&\subset \left(\kappa \alpha_-,\kappa^{-3}\alpha_+\right).
\end{alignat}
\end{subequations}
It follows from \eqref{e4a} that $\left(\kappa^{-1}x_0m^{-1},\kappa x_0 m^{-1}\right)\cap \left(\gm,\gp\right)\supset \{x_0m^{-1}\}\neq \emptyset$. So, since the set $\Tgg$ is dense in $\left[\gm,\gp\right]$, there exist $M\in \n$ and $t_1, \ldots , t_M \in T^{\ast}$ such that
\begin{align}
t_1 \cdot\ldots\cdot t_i&\in \left[\gm,\gp\right], \quad i=1, \ldots, M,\label{e5}\\
t_1 \cdot\ldots\cdot t_M&\in \left(\kappa^{-1}x_0 m^{-1},\kappa x_0 m^{-1}\right).\nonumber
\end{align}
Multiplying $t_1 \cdot\ldots\cdot t_{i-1}$ for $i=2, \ldots M$ by the inverse of $t_1 \cdot\ldots\cdot t_M$ we come to
\begin{subequations}
\begin{alignat}{4}
\label{e6a}
t^{-1}_M \cdot\ldots\cdot t^{-1}_i&\in \left(\kappa^{-1}\gm x^{-1}_0 m,\kappa \gp x^{-1}_0m\right), \quad i=2, \ldots, M,\\
\label{e6b}
t^{-1}_M \cdot\ldots\cdot t^{-1}_1&\in \left(\kappa^{-1}x^{-1}_0 m,\kappa x^{-1}_0 m\right).
\end{alignat}
\end{subequations}

Similarly, in view of the second inequality \eqref{e3b}, there exist $N \in \n$ and $u_1, \ldots, u_N \in T^{\ast}$ such that
\begin{subequations}
\begin{alignat}{4}\label{e7a}
u_1 \cdot\ldots\cdot u_i&\in \left[\gm,\gp\right], \quad i=1, \ldots, N,\\
\label{e7b}
u_1 \cdot\ldots\cdot u_N&\in \left(\kappa,\kappa^2\right).
\end{alignat}
\end{subequations}
Put
\[
l:=\left(t^{-1}_M, t^{-1}_{M-1}, \ldots , t^{-1}_1,u_1,\ldots, u_N, t_1, \ldots, t_M\right).
\]

Now we prove that 
\begin{eqnarray}\label{e8}
x_0l_1 \cdot\ldots\cdot l_i\in \left(\alpha_-,\alpha_+\right), \qquad i=1, \ldots, 2M+N.
\end{eqnarray}
If $i \in \{1, \ldots, M-1\}$ then, by \eqref{e6a} and \eqref{e4b}, we have
\begin{eqnarray*}
x_0l_1 \cdot\ldots\cdot l_i\in m\left(\kappa^{-1}\gm,\kappa\gp\right)\subset \left(\kappa^{-1}\kappa \alpha_-,\kappa\kappa^{-3}\alpha_+\right)\subset \left(\alpha_-,\alpha_+\right).
\end{eqnarray*}
When $i=M$ conditions \eqref{e6b}, \eqref{e3b} and \eqref{e4b} implies
\begin{eqnarray*}
x_0l_1 \cdot\ldots\cdot l_i\in m\left(\kappa^{-1},\kappa\right)\subset m\left(\gm,\gp\right)\subset \left(\kappa\alpha_-,\kappa^{-3}\alpha_+\right)\subset \left(\alpha_-,\alpha_+\right).
\end{eqnarray*}
In the case when $i \in \{M+1, \ldots, M+N\}$, making use of \eqref{e6b}, \eqref{e7a} and \eqref{e4b}, we get
\begin{eqnarray*}
x_0l_1 \cdot\ldots\cdot l_i\in m\left(\kappa^{-1}\gm,\kappa\gp\right)\subset \left(\kappa^{-1}\kappa\alpha_-,\kappa\kappa^{-3}\alpha_+\right)\subset \left(\alpha_-,\alpha_+\right).
\end{eqnarray*}
Finally, if $i \in \{M+N+1, \ldots , 2M+N\}$ then \eqref{e6b}, \eqref{e7b}, \eqref{e5} and \eqref{e4b} force 
\begin{eqnarray*}
x_0l_1 \cdot\ldots\cdot l_i\in m\left(\kappa^{-1}\kappa\gm,\kappa \kappa^2 \gp\right)= m\left(\gm,\kappa^3\gp\right)\subset \left(\kappa\alpha_-,\kappa^3\kappa^{-3}\alpha_+\right)\subset \left(\alpha_-,\alpha_+\right)
\end{eqnarray*}
which completes the proof of \eqref{e8}. \\

Since $x_0\in \Taa$ and $l_1, \ldots, l_{2M+N}\in T^{\ast}$ it follows from \eqref{e8} that
\[
x_0l_1 \cdot\ldots\cdot l_{2M+N}\in \Taa.
\]
On the other hand, using the definition of $l$, the relation $x_0\in \left(\kappa^{-1}a,a\right]$ and \eqref{e7b}, we get
\begin{eqnarray*}
x_0l_1 \cdot\ldots\cdot l_{2M+N}= x_0u_1\cdot\ldots\cdot u_N \in \left(\kappa^{-1}a\kappa,a \kappa^2\right)=\left(a,\kappa^2a\right).
\end{eqnarray*}
This implies that $x_0l_1 \cdot\ldots\cdot l_{2M+N} \in\left(a,\kappa^2a\right)\cap\Taa$ contrary to \eqref{e3a}.
\end{proof} 

Using the above lemma we conclude that for every nonempty $T\subset(0,+\infty)$ there exists a limit ratio $\gp/\gm$ such that $\Tgg$ is dense in $\left[\gm,\gp\right]$. This is in fact one of the most important result of this note.

\begin{thm}[Limit ratio principle]\label{t1.2} 
For every $T\subset \zn$ the formula 
\begin{eqnarray*}
R_T=\inf \left\{\frac{\gp}{\gm}: \gm \in (0,1), \gp \in \jn\mbox{ and } \Tgg \mbox{ is dense in } \left[\gm,\gp\right] \right\}
\end{eqnarray*}
defines a unique $R_T \in [1,+\infty]$ such that\\
\noindent
$(i)$ for all $\gm \in (0,1)$ and $\gp \in \jn$ with $\gp/\gm>R_T$ the set $\Tgg$ is dense in $\left[\gm,\gp\right]$,\\
\noindent
$(ii)$ for all $\gm \in (0,1)$ and $\gp \in \jn$ with $\gp/\gm<R_T$ the set $\Tgg$ is not dense in $\left[\gm,\gp\right]$.
\end{thm}
\noindent
The value $R_T$ is called the {\it limit ratio of $T$}.\\

The remaining case $\gp/\gm=R_T$ is not covered by the present note. In this special case it is not known if the situation depends on the ratio only. Moreover, the gap between necessary and sufficient conditions motivates us to distinguish the assumption $\gp/\gm>R_T$ and the density of $\Tgg$. It turns out that  the first one is suitable in most cases. 

In the rest of this section we deliver several properties of the limit ratio. Below we adopt the convention that $\inf \emptyset=+\infty$.

\begin{prop}
 The limit ratio operator $R$ is decreasing in the inclusion ordering, that is $R_{T}\geq R_{V}$ for all $T$, $V$ with $T\subset V\subset \zn$.
\end{prop}
\begin{proof}
Fix $T$ and $V$ as above. Then, direcly from the definition, we know that $\Tggn \subset V_{\gm,\gp,n}$ for all $\gm \in (0,1)$, $\gp \in (1,+\infty)$ and $n \in \mathbb{N}$. Therefore 
$\Tgg \subseteq V_{\gm,\gp}$ for all $\gm \in (0,1)$ and $\gp \in (1,+\infty)$. 

In particular, if $\Tgg$ is dense in $[\gm,\gp]$ for some $\gm,\ \gp$ as above, then so is $V_{\gm,\gp}$. Thus by the limit ratio principle we immediately obtain $R_T \ge R_V$.
\end{proof}

\begin{prop}\label{p1.3}
Given a nonempty set $T \subset \zn$ we have what follows:
\begin{enumerate}[(i)]
\item \label{R.7} $R_T\geq \inf\left(T^{\ast}\cap \jn\right)$;
\item \label{R.6} $R_{T^{\ast}}= R_{T}$;
\item \label{R.1} if $T=\{a,b\}$ for some $a \in (0,1)$, and $b \in \jn$ with $\log a/\log b\notin \mq$, then $R_T=b/a$;
\item \label{R.2} $R_T= 1$ if and only if $1$ is the accumulation point of $T$;
\item \label{R.8} if $T^*$ has an accumulation point $c \in [1,+\infty)$, then $R_T \leq c^2$.
\end{enumerate}
\end{prop}
\begin{proof} 
If $T=\{1\}$ then all these properties are trivial. Thus hereafter we assume that $T$ contains at least one element different than $1$.

\eqref{R.7} Clearly $T^* \cap (1,+\infty) \ne \emptyset$. Suppose that $R_T<\inf(T^*\cap(1,+\infty))$ and let $\gm\in(0,1)$ and $\gp \in (1,+\infty)$ be such that 
 $$\frac{\gp}{\gm} \in \Big( R_T, \inf\big(T^*\cap(1,+\infty)\big)\Big).$$
Then $\Tgg$ is dense in $[\gm,\gp]$. Take any $t \in \Tgg \cap (1,+\infty)$ and choose $n \in \n$ and $t_1,\ldots,t_n \in T^*$ satisfying the conditions $t=t_1\cdot\ldots\cdot t_n$ and $\gm \le t_1\cdot\ldots\cdot t_k \le \gp$ for all $k =1,\ldots,n$. Since $t>1$ then $t_k>1$ for at least one $k=1,\dots,n$. If $k=1$ then $t_1 \in T^* \cap (1,+\infty)$ hence 
$$t_1>\frac{\gp}{\gm}>\gp$$
which is impossible. If $2\le k\le n$ then $t_k \in T^*\cap(1,+\infty)$, so 
$$t_1\cdot\ldots\cdot t_k=(t_1\cdot\ldots\cdot t_{k-1})\cdot t_k>\gm \frac{\gp}{\gm}=\gp,$$
that is false again. Consequently, the assertion \eqref{R.7} follows.

To notice \eqref{R.6} observe that $(T^*)^*=T^*$ hence $T^*_{\gm,\gp}=\Tgg$ and finally $R_{T^*}=R_T$. Assertion \eqref{R.1} immediately follows from \cite[Proposition~2]{JP}.

Assuming $R_T=1$ and using assertion \eqref{R.7} we see that $\inf\left(T^{\ast}\cap \jn\right)=1$ which implies that $1$ is an accumulation point of $T^*$. Taking any sequence $(u_k)_{k \in \mathbb{N}}$ of elements in $T^* \setminus \{1\}$ converging to $1$ and putting
$$ v_k=\begin{cases}
        u_k, &\text{ if }u_k \in T,\\
        1/u_k, &\text{ if }u_k \in T^{-1},
       \end{cases}
$$
we get a sequence $(v_k)_{k \in \n}$ of points in $T \setminus\{1\}$, tending to $1$. Consequently, $1$ is an accumulation point of the set $T$. 

In order to prove the converse implication assume that $1$ is an accumulation point of $T$. To show the equality $R_T=1$ it is enough to check that the set $\Tgg$ is dense in $[\gm,\gp]$ for arbitrary $\gm\in(0,1)$ and $\gp \in(1,+\infty)$. Fix such $\gm,\gp$ and take $a,b \in \mathbb{R}$ satisfying $\gm\le a<b \le \gp$. We prove that $\Tgg \cap (a,b)\ne \emptyset$. Since $1 \in \Tgg$ and $(T^*)^{-1}=T^*$ we may assume that $1\le a < b$. Moreover, as $1$ is also an accumulation point of $T^*$, the equality $(T^*)^{-1}=T^*$ implies that $1$ is an accumulation point of $T^*\cap(1,\gp)$, and thus we can find $t \in T^* \cap(1, b/a)$. Let $N\in \mathbb{N}$ be the maximal positive integer such that $t^N<b$.
Then
$$1<t<t^2<\ldots<t^N<b \le \gp,$$
hence $t^N \in T_{\gm,\gp,N}$. On the other hand $t^N=\frac{t^{N+1}}t>\frac{b}{t}>a$. Consequently, $t^N \in \Tgg \cap (a,b)$. Thus we have proved \eqref{R.2}.

Finally we proceed to prove \eqref{R.8}. In the case $c=1$ it is a direct consequence of \eqref{R.2}. From now on assume that $T^*$ has an accumulation point $c \in(1,+\infty)$ and suppose to the contrary that there exists $\gm\in(0,1)$ and $\gp \in(1,+\infty)$ such that $\frac{\gp}{\gm}>c^2$ and $\Tgg$ is not a dense subset of $[\gm\gp]$. Then, since $1 \in \Tgg$, we have four cases:
\begin{enumerate}
 \item $\cl \Tgg=[\gm,\delta]$ for some $\delta \in[1,\gp)$;
 \item $\cl \Tgg=[\delta,\gp]$ for some $\delta \in(\gm,1]$;
 \item $\cl \Tgg=[a,b]$ for some $a \in (\gm,1]$ and $b \in [1,\gp)$;
 \item $\cl \Tgg$ is not an interval;
\end{enumerate}

At first consider the case when $\cl \Tgg=[\gm,\delta]$ for some $\delta \in[1,\gp)$. Since $c$ is an accumulation point of $T^*$ and $1<c<\frac\gp\gm$ we can find an element $t \in T^* \cap (1,\frac{\gp}{\gm})$.
Then $\frac \delta t < \delta$ and $\gm< \frac{\gp}t$. In particular, the intervals $( \gm,\delta)$ and $(\frac{\delta}t,\frac{\gp}t)$ have a nonempty intersection. Since $\Tgg$ is a dense subset of $[\gm,\delta]$, there exists an element $v \in \Tgg \cap (\frac\delta t,\frac{\gp}t)$. 
Next, since $t \in T^*$, $v \in \Tgg$, and $tv \in (\delta,\gp) \subset (\gm,\gp)$, we have $tv \in \Tgg$. But it is a contradiction since $tv>\delta$ and $\cl \Tgg=[\gm,\delta]$. Analogously we can exclude the case $\cl \Tgg=[\delta,\gp]$ for some $\delta \in(\gm,1]$. 

In the third case we have 
$\frac{\gp}{\gm}\frac ba > c^2$, and hence either $\frac{\gp}a>c$ or $\frac{b}{\gm}>c$. 
In the first subcase there exists $t \in T^* \cap (1,\frac{\gp}a)$. Using the definition of $\Tgg$ one can verify that $t\Tgg \cap [\gm,\gp]\subset \Tgg$, and thus 
$$\cl T_{\gm,\gp} \supset \cl t T_{\gm,\gp} \cap [\gm,\gp]=[ta,tb]\cap[\gm,\gp].$$
But, since $ta <\gp$ and $tb>1>\gm$ we have $[ta,tb]\cap(\gm,\gp) \ne \emptyset$, and thus
$\sup T_{\gm,\gp} \ge \min\{tb,\gp\}>b$,
which is a contradiction. The prove in the second subcase in analogous.

The only remaining case is that $\Tgg$ has a gap (case 4). More precisely there exists $a,b \in (\gm,\gp)$ with $a<b$ such that the interval $(a,b)$ is disjoint with $\Tgg$, whenever both sets $(\gm,a]\cap \Tgg$ and $[b,\gp)\cap \Tgg$ are nonempty. Let us assume that $(a,b)$ is a maximal open interval with this property. Then we have $\frac{ac}{bc^{-1}}=\frac ab c^2< c^2 < \frac{\gp}{\gm}$, in particular either $ac<\gp$, or $bc^{-1}>\gm$. This splits the proof of this case into two subcases (with the same splitting condition as in case~3).

At first let us assume that $ac<\gp$. Since $c \in (1,\frac{\gp}a)$ is an accumulation point of $T^*$, the set 
$$\Gamma:=T^* \cap \Big(1,\frac \gp a\Big) \cap \Big(c \sqrt{\frac ab},c \sqrt{\frac ba}\Big)$$ is infinite. Take two elements $c_1,c_2 \in \Gamma$ with $c_1<c_2$. 

Since $(a,b)$ is a maximal interval disjoint with $\Tgg$ and $(\gm,a]\cap \Tgg$ is nonempty, there exists an element $u \in \Tgg \cap (a\frac{c_1}{c_2},a]$. We are going to show that $uc_2c_1^{-1} \in \Tgg$. 
First observe that $(a,ac_2]\subset(a,\gp)\subset(\gm,\gp)$. Thus 
\begin{align*}
 uc_2 &\in (ac_1,ac_2] \subset (a,ac_2] \subset (\gm,\gp),\\
 uc_2c_1^{-1} &\in (a,a\tfrac{c_2}{c_1}] \subset (a,ac_2] \subset (\gm,\gp),
\end{align*}
and therefore, as $u \in \Tgg$, we get  $uc_2c_1^{-1} \in \Tgg$. 

Moreover, since $c_1,c_2$ belong to $\Gamma$ we get
$$a\frac{c_2}{c_1}<a\frac{c\:\sqrt{\frac ba}}{c\:\sqrt{\frac ab}}=b.$$
Finally we get $uc_2c_1^{-1} \in (a,a\frac{c_2}{c_1}] \subset (a,b)$, which contradicts the choice of $a,b$.

In the second subcase, that is if $bc^{-1}>\gm$, the corresponding set $\Gamma$ is defined as
$T^* \cap \big(\frac \gm b,1\big) \cap \big(c^{-1} \sqrt{\frac ab},c^{-1} \sqrt{\frac ba}\big)$, we take elements $c_1,c_2 \in \Gamma$ with $c_2<c_1$, and $u \in \Tgg \cap (b,b\frac{c_1}{c_2})$. Then we come to a contradiction showing again that $u c_2c_1^{-1} \in \Tgg \cap(a,b)$.
\end{proof}

\begin{rem}
 The number $c^2$ appearing on the right hand side of the property~\eqref{R.8} cannot be diminished. Indeed, for $c \in \jn$,  $T=\{c\cdot2^{1/k} \colon k \in \mathbb{N}_+\}$ and $\gp \in (c,c^2)$ we obtain 
$$ \Big(\frac{\gp}c,c\Big) \cap [t,t\gp) =\emptyset \text{ for all }t \in T^*,$$
which implies $(\tfrac{\gp}c,c) \cap T_{1,\gp} =\emptyset$. Thus $R_T = c^2$.
\end{rem}

Now we are in the position to characterize the finiteness of the limit ratio. In fact we present a necessary and sufficient condition to the equality $R_T=+\infty$, as it is easier to formulate.

\begin{prop}\label{p1.i} Let $T \subset \zn$. Then $R_T=+\infty$ if and only if there exists $t \in \zn$ such that
    $T$ is a discrete subset of $\{t^q \colon q \in \mq\}$. 
\end{prop}
\begin{proof}
Assume first that $R_T=+\infty$. Then, by Proposition~\ref{p1.3}.\eqref{R.8}, the set $T^*$ has no accumulation points. Next, take any $t_0 \in T^* \setminus\{1\}$. Then, by Proposition~\ref{p1.3}.\eqref{R.1}, since $R_T=+\infty$, we get $\frac{\log t}{\log t_0} \in \mq$ for all $t \in T$. Consequently $T \subset \{t_0^q\colon q \in \mq\}$.

To prove the converse implication assume that $T$ is a discrete subset of $\{t^q \colon q \in \mq\}$ for some $t>0$ and let $\vartheta>1$. Then the set $T \cap [\vartheta^{-1},\vartheta]$ is a finite subset of $\{t^q \colon q \in \mq\}$,
say $T\cap [\vartheta^{-1},\vartheta]=\{t^{\frac{p_1}r},\dots,t^{\frac{p_n}r}\}$ where $p_1,\dots,p_n \in \mathbb{Z}$ and $r \in \mathbb{N}$. Then, defining $t_\vartheta:=t^{\frac 1r}$,  we have
$$(\vartheta^{-1}, \vartheta) \cap T^* \subset \big\{t_\vartheta^n \colon n \in \mathbb{Z}\big\}.$$
Therefore $T_{\vartheta^{-1}, \vartheta} \subset \{t_\vartheta^n \colon n \in \mathbb{Z}\}$ is not a dense set, which by limit ratio principle implies $R_T \ge \vartheta^2$.
As $\vartheta$ is an arbitrary number in $\jn$, we obtain $R_T=+\infty$.
\end{proof}

\section{Extension procedure} 
In what follows the phrase {\it nontrivial interval} means {interval with different endpoints}. Let us also remind that we adopt the convention that $a/0=+\infty$ for all $a \in (0,+\infty]$. Moreover, below $e^n$ denotes the $n$-th iterate of the function $e$.

\begin{lem}[Extension lemma]\label{l2.1} 
Let $T\subset \zn$ be a nonempty set and $I\subset \zn$ be an interval with $\sup I/\inf I>R_T$.
Then, for every nontrivial interval $U\subset I$, 
\begin{eqnarray}\label{e12}
I=\bigcup^{\infty}_{n=0} e^n(U)
\end{eqnarray}
where $e$ is the selfmapping of $2^I$ given by
\begin{eqnarray}\label{e14}
e(A)=\left(T^{\ast}\cdot A\right)\cap I.
\end{eqnarray}
\end{lem}
\begin{proof} The assumption $\sup I/ \inf I>R_T$ implies that $\sup I>\inf I$, so the interval $I$ is nontrivial. 
Fix a nontrivial interval $U\subset I$. At first assume that $I$ is open, $\inf I>0$ and $\sup I<+\infty$. Take any nontrivial open interval $S\subset U$ such that $R_T \sup S/\inf S<\sup I/\inf I$, also $\inf I<\inf S$ and $\sup S<\sup I$. Then $0<\inf S< \sup S<+\infty$ and 
\[
\frac{\sup I/\sup S}{\inf I/\inf S}=\frac{\sup I/\inf I}{\sup S/\inf S}>R_T.
\]
Put
\[
\gm=\frac{\inf I}{\inf S}  \quad \mbox{and} \quad \gp= \frac{\sup I}{\sup S}.
\]
Then $\gm\in (0,1)$, $\gp\in \jn$ and $\gp/\gm>R_T$, and thus the set $\Tgg$ is dense in $\left[\gm,\gp\right]$ according to Theorem~\ref{t1.2}. Therefore, since 
\begin{eqnarray}\label{e9}
I=S\cdot\left[\frac{\inf I}{\inf S},\frac{\sup I}{\sup S} \right]=S\cdot \left[\gm,\gp\right]
\end{eqnarray}
and the set $S$ is open, we have
\begin{eqnarray}\label{e10}
I=S \cdot \Tgg.
\end{eqnarray}

Now, using induction, we prove that
\begin{eqnarray}\label{e11}
s\Tggn \subset e^n\left(\left\{s\right\}\right), \quad s\in S,
\end{eqnarray}
for all $n \in \n$. Observe that \eqref{e9} forces the inclusion
\[
s\Tggn \subset I, \quad s\in S, n \in \n.
\]
Thus 
\[
sT_{\gm,\gp,1} \subset \left(sT^{\ast}\right)\cap I=e\left(\left\{s\right\}\right)=e^{1}\left(\left\{s\right\}\right), \quad s\in S,
\]
which is \eqref{e11} for $n =1$. Fix a positive integer $n$ and assume \eqref{e11}. Take any $s \in S$ and $u \in T_{\gm,\gp,n+1}$. Then there exist $v \in \Tggn$ and $t \in T^{\ast}$ such that $u=vt$. Making use of \eqref{e11} we have $sv \in e^n\left(\left\{s\right\}\right)$. Moreover, \eqref{e9} implies the relation $su \in I$, and thus 
\begin{eqnarray*}
su=svt\in \left(svT^{\ast}\right)\cap I=e\left(\left\{sv\right\}\right) \subset e\left(e^n\left(\left\{s\right\}\right)\right)=e^{n+1}\left(\left\{s\right\}\right).
\end{eqnarray*}
Consequently, \eqref{e11} holds for all $n \in \n$.\\

It follows from \eqref{e10} and \eqref{e11} that
\begin{eqnarray*}
I=S\cdot \Tgg=\bigcup^{\infty}_{n=1}\bigcup_{s \in S} s\Tggn \subset \bigcup^{\infty}_{n=1}\bigcup_{s \in S} e^n\left(\left\{s\right\}\right)=
\bigcup^{\infty}_{n=1}e^n(S) \subset \bigcup^{\infty}_{n=0}e^n(U).
\end{eqnarray*} 
As the converse inclusion is obvious, we obtain the desired equality \eqref{e12}.\\

If $I$ is an arbitrary open interval contained in $\zn$ and $U\subset I$ is a nontrivial interval, then, putting $I_k=I\cap(1/k,k)$ and $U_k=U\cap(1/k,k)$ for all $ k \in \n$ with $k \ge R_T$, and applying the just proved assertion we have
\[
I_k=\bigcup^{\infty}_{n=0}e^n_k\left(U_k\right), \quad k \in \n,\ k \ge R_T,
\]
where $e_k$, mapping $2^{I_k}$ into itself, is given by 
\[
e_k(A)=\left(T^{\ast}\cdot A\right)\cap I_k.
\]
Hence
\[
I=\bigcup^{\infty}_{k=\ceil{R_T}}I_k=\bigcup^{\infty}_{k=\ceil{R_T}}\bigcup^{\infty}_{n=0}e^n_k\left(U_k\right)
\subset \bigcup^{\infty}_{n=0}\bigcup^{\infty}_{k=\ceil{R_T}}e^n\left(U_k\right)\subset \bigcup^{\infty}_{n=0}e^n\left(U\right) 
\]
and thus \eqref{e12} follows.\\

Finally consider the case  when a nontrivial interval $I$ contains at least one its endpoints.  Denote by $e_0$ the multifunction $e$ relativised to the interval $\w   I$, that is $e_0$ maps $2^{\w   I}$ into itself and is given by 
\[
e_0(A)=\left(T^{\ast}\cdot A\right)\cap \w   I.
\]
For any $A\subset \w   I$ we have
\[
e_0(A)=\left(T^{\ast}\cdot A\right)\cap \w   I\subset \left(T^{\ast}\cdot A\right)\cap I=e(A).
\]
Now, if $U\subset I$ is a nontrivial interval, then $\w   U\subset \w   I$, and thus, applying the just proved part of the lemma, we have
\[
\w   I =\bigcup^{\infty}_{n=0}e_0^n \left(\w   U\right)\subset\bigcup^{\infty}_{n=0}e^n \left(\w   U\right) \subset\bigcup^{\infty}_{n=0}e^n \left(U\right),
\]
that is
\begin{eqnarray}\label{en}
\w   I\subset \bigcup^{\infty}_{n=0}e^n \left(U\right).
\end{eqnarray}
Put $\alpha=\inf I$, $\beta=\sup I$ and assume for instance that $\alpha \in I$. If $T^{\ast}\cap \left(1,\beta/\alpha\right)=\emptyset$ then, using induction, one could prove that $T_{1,\beta/\alpha,n}\subset\left\{1,\beta/\alpha\right\}$ for all $n \in \n$, and thus 
$T_{1,\beta/\alpha}\subset\left\{1,\beta/\alpha\right\}$.
In particular, Theorem~\ref{t1.2} would yield 
$$R_T \ge \frac\beta\alpha=\frac{\sup I}{\inf I}>R_T,$$ a contradiction.

 Consequently, we can find $t_0\in T^{\ast}\cap \left(1,\beta/\alpha\right)$. Then $t_0\alpha \in \w   I$ and, by  \eqref{en}, we see that $t_0\alpha\in e^N(U)$ for some $N\in \n$. Therefore
\[
\alpha\in t^{-1}_0e^N(U)\cap I\subset \left(T^{\ast}\cdot e^N(U)\right)\cap I=e^{N+1}(U).
\]
Analogously we can prove that if $\beta \in I$ then $\beta\in \bigcup^{\infty}_{n=0}e^n \left(U\right)$, so the proof is completed.
\end{proof}

\begin{thm}[Extension theorem]\label{t2.2}
Let $T\subset \zn$ be a nonempty set and $I\subset \zn$ be an open interval with $\sup I/\inf I>R_T$. Let $J\subset \mr$ be an open interval and $\varphi:I\ra J$ be a solution of the simultaneous equations
\begin{eqnarray}\label{e13}
\varphi(tx)=g_t\left(x,\varphi(x)\right), \quad t \in T^{\ast}, 
\end{eqnarray}
postulated for $x \in I\cap t^{-1}I$, where $(g_t)_{t \in T^{\ast}}$ are continuous [analytic] functions mapping $I\times J$ into $J$. If there exists a nontrivial interval $U\subset I$ such that $\varphi|_{U}$ is continuous [analytic], then so is $\varphi$. 
\end{thm}
\begin{proof}
It follows from the assumptions that $\sup I>\inf I$, so the interval $I$ is nontrivial. Define the mapping $e:2^I\ra 2^I$ by \eqref{e14}, that is 
\[
e(A)=\bigcup_{t \in T^{\ast}}tA\cap I, \quad A\subset I.
\]
Now \eqref{e13} implies that if $\varphi$ is continuous [analytic] on a nontrivial interval $V\subset I$, then it is continuous [analytic] on $e(V)$ which is an open set containing $V$. Thus, applying this procedure iteratively, starting from the interval $U$ and making use of Lemma \ref{l2.1}, we see that $\varphi$ is continuous [analytic] on $I$.
\end{proof}

In fact we can adapt this proof to other localizable properties such as measurability, differentiability, monotonicity, convexity (concavity), of course under suitable assumptions imposed on given functions $g_t$. 
Recall here that a property (P) of a function $\varphi \colon I \to \mathbb{R}$ is \emph{localizable} if (P) holds for~$f$ whenever (P) holds in an open neighbourhood of every point in $I$.

\section{Application to simultaneous equations \eqref{e2}}

Now we are in the position to strengthen Theorem~JP essentially. Assuming continuity of the solution of system  \eqref{e2} on an arbitrary nontrivial interval we determine the form of the solution on the whole domain. 

\begin{thm}\label{t3.1}
Let $T\subset \zn$ be a nonempty set, $c:T\ra \mr$ and $I\subset \zn$ be an interval with $\sup I/\inf I>R_T$. Let $\varphi:I\ra \mr$ be a solution of equations~\eqref{e2}, which is continuous  on a nontrivial subinterval of~$I$. \\
\noindent
$(i)$ If $p\neq 0$ then there exist $a,b \in \mr$ such that
\[
\varphi(x)=ax^p+b, \quad x \in I.
\] 
\noindent
$(ii)$ If $p= 0$ then there exist $a,b \in \mr$ such that
\[
\varphi(x)=a\log x+b, \quad x \in I.
\] 
\end{thm}
\begin{proof}
We start with an extension of $c$ to the set $T^{\ast}$. If $1 \in T$ then \eqref{e2} forces  that $c(1)=0$. If $1 \notin T$ then define $c(1):=0$. Moreover, put 
\[
c(t):=-t^pc\left(t^{-1}\right)
\] 
for all $t \in T^{-1}$ and observe that whenever $t \in T\cap T^{-1}$ the above equality follows from \eqref{e2}. It is an easy task to check that $\varphi$ is a solution of the simultaneous equations
\[
\varphi(tx)=\varphi(x)+c(t)x^p, \quad t \in T^{\ast},
\]  
postulated for $x \in I\cap t^{-1} I$ (cf. also the second paragraph of the proof of \cite[Proposition 1]{JP}).
Now, on account of Lemma \ref{l2.1}, the function $\varphi$ is continuous on $I$.

Choose any $\gm \in (0,1)$ and $\gp\in \jn$ such that $\sup I/\inf I> \gp/\gm >R_T$. Then
\[
\frac{\inf I}{\gm}<\frac{\sup I}{\gp}
\]
and, by the definition of $R_T$, the set $\Tgg$ is dense in $\left[\gm, \gp\right]$. Therefore Theorem~JP implies \eqref{e16} when $p\neq 0$ and \eqref{e17} in the case $p=0$. This forces that $\varphi$ is analytic on the interval $J:=\left({\inf I}/{\gm},{\sup I}/{\gp}\right)$, and thus, by Theorem \ref{t2.2}, on the interior of $I$. 
Therefore the formula of the function $\varphi$ in the interior of $I$ is the same as that of its restriction $\varphi|_{J}$. Moreover, since $\varphi$ is continuous, it is also valid for every $x \in I$.
%
\end{proof}

This theorem has the important corollary in the case $c \equiv 0$.
\begin{cor}
Let $T\subset \zn$ be a nonempty set, and $I\subset \zn$ be an interval with $\sup I/\inf I>R_T$. Let $\varphi:I\ra \mr$ be a solution of equations 
$$
\varphi(tx)=\varphi(x), \qquad t \in T,
$$
which is continuous  on a nontrivial subinterval of~$I$. Then $\varphi$ is constant.
\end{cor}
\begin{proof}
 Applying Theorem~\ref{t3.1} with $c \equiv 0$ and $p\in\{1,2\}$ we obtain $\varphi(x)=a_1x+b_1$ and $\varphi(x)=a_2 x^2+b_2$ for some $a_1,b_1,a_2,b_2 \in \mr$.
 
 Thus $a_1x+b_1$ and $a_2x^2+b_2$ are the same polynomials, which implies $a_1=a_2=0$ and $b_1=b_2$. Finally we get $\varphi \equiv b_1$.
\end{proof}

At the very end of the paper we show that our key assumption ${\sup I}/{\inf I}>R_T$ is essential and not only technical.

\begin{ex}\label{ex3.1}
Take $I=\left[1/2,5/2\right]$ and $T=\left\{1/2,3\right\}$. Then $\sup I/\inf I=5$ and (cf. Proposition~\ref{p1.3} \eqref{R.1}) $R_T=6$, so $\sup I/\inf I<R_T$. Put also $g_t(x,y)=y$ for $t\in T$, $x\in I$ and $y \in \mr$. Now the system \eqref{e2} takes the form
\begin{eqnarray}\label{e18}
\left\{\begin{array}{ll}
\varphi\left(\frac{1}{2}x\right)=\varphi(x),\\
\varphi\left(3x\right)=\varphi(x).
\end{array}\right.
\end{eqnarray} 
The first equation is postulated for $x \in I\cap \left(\frac{1}{2}\right)^{-1}I=\left[\frac{1}{2},\frac{5}{2}\right]\cap\left(2\cdot\left[\frac{1}{2},\frac{5}{2}\right]\right)=\left[\frac{1}{2},\frac{5}{2}\right]\cap\left[1,5\right]=\left[1,\frac{5}{2}\right]$, whereas the second one for $x \in I\cap 3^{-1}I=\left[\frac{1}{2},\frac{5}{2}\right]\cap\left(\frac{1}{3}\cdot\left[\frac{1}{2},\frac{5}{2}\right]\right)=\left[\frac{1}{2},\frac{5}{2}\right]\cap \left[\frac{1}{6},\frac{5}{6}\right]=\left[\frac{1}{2},\frac{5}{6}\right]$. Equations \eqref{e18} generate also two others for $t \in T^{-1}=\left\{\frac{1}{3},2\right\}$:
\begin{align*}
\begin{cases}
\varphi\left(\frac{1}{3}x\right)=\varphi(x),\\
\varphi\left(2x\right)=\varphi(x), 
\end{cases} 
\end{align*}

and thus $g_t(x,y)=y$, $x \in I$, $y \in \mr$, also for $t \in T^{-1}$.
We  construct a family of continuous solutions which are continuous but not analytic, possibly even not differentiable. Let $s: [0,1]\ra \mr$ be any continuous function satisfying $s(0)=s(1)=0$. For arbitrary $\alpha,\beta \in \mr$ with $\alpha<\beta$ define a function $Q_{\alpha,\beta}:\mr \ra \mr$  by 
\begin{eqnarray*}
Q_{\alpha,\beta}(x)=\left\{\begin{array}{ll}
s\left(\frac{\beta-x}{\beta-\alpha}\right), \quad &\mbox{if } x \in \left(\alpha,\beta\right),\\
0, \quad &\mbox{otherwise. }
\end{array}\right.
\end{eqnarray*} 
Any such a function is continuous and satisfies the condition
\[
Q_{\alpha,\beta}(cx)=Q_{\alpha/c,\beta/c}(x), \quad x \in \mr.
\]
Using the last property one can easily check that the function
\[
\varphi_0:=Q_{5/8,2/3}+Q_{15/16,1}+Q_{5/4,4/3}+Q_{15/8,2}
\]
is a solution of system \eqref{e18}. 

\begin{figure}
 \begin{center}
\begin{tikzpicture}[xscale=5]
\draw (0.5,0)--(0.625,0);
\draw (0.625,0) parabola[parabola height=1cm] +(0.04166,0);
\draw (0.6666666,0)--(0.9375,0);
\draw (0.9375,0) parabola[parabola height=1cm] +(0.0625,0);
\draw (1,0) --(1.25,0);
\draw (1.25,0) parabola[parabola height=1cm] +(0.0833333,0);
\draw (1.333333,0) -- (1.875,0);
\draw (1.875,0) parabola[parabola height=1cm] +(0.125,0);
\draw (2,0)--(2.5,0);
\node at (0.5,-0.3) {$\tfrac12$};
\node at (0.625,-0.3) {$\tfrac58$};
\node at (0.666666666,-0.3) {$\tfrac23$};
\node at (0.9375,-0.3) {$\tfrac{15}{16}$};
\node at (1,-0.3) {$1$};
\node at (1.25,-0.3) {$\tfrac54$};
\node at (1.333333333,-0.3) {$\tfrac43$};
\node at (1.5,-0.3) {$\tfrac32$};
\node at (1.875,-0.3) {$\tfrac{15}8$};
\node at (2,-0.3) {$2$};
\node at (2.5,-0.3) {$\tfrac52$};
\draw[style=dashed] (0.65833333,1.1) to [bend left=45] (1.3166666666,1.2);
\draw[style=loosely dashed] (0.658333,1.14) to [bend left=75] (1.9375,1.14);
\draw[style=dashed] (0.9875,1.1) to [bend left=65] (1.9375,1.1);
\end{tikzpicture}
\caption{\small{Plot of the function $\varphi_0$. All assertions implied by the system \eqref{e18} are presented by dashed lines.}}
\end{center}
\end{figure}
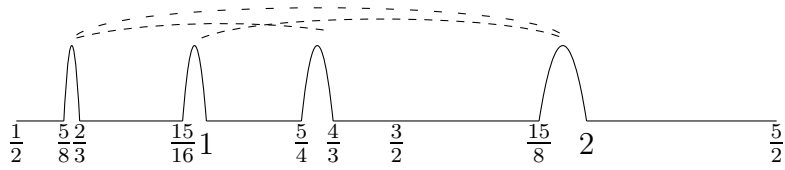
\end{ex}

\begin{small}\noindent (W. Jarczyk) 
{\sc Institute of Mathematics, 
University of Zielona Gora, 
prof. Z.~Szafrana~4a, 
PL-65-516 Zielona G\'ora, 
Poland \\
e-mail: {\tt w.jarczyk@im.uz.zgora.pl}
}

\bigskip\noindent (P. Pasteczka) 
{\sc  Institute of Mathematics,
University of the National Education Commission,
Podchor\k{a}\.zych 2, 
PL-30-084 Krak\'ow,
Poland \\
e-mail: {\tt pawel.pasteczka@up.krakow.pl}} 
\end{small}

\end{document}